\documentclass[preprint,12pt]{elsarticle}
\usepackage{amsmath}
\usepackage{amssymb}
\usepackage{amsthm}
\usepackage{graphicx}
\usepackage{listings}
\usepackage{float}
\usepackage{verbatim}
\usepackage{color}

\newcommand{\e}{\text{e}}

\newtheorem{theorem}{Theorem}
\newtheorem{lemma}[theorem]{Lemma}
\newtheorem{corollary}[theorem]{Corollary}
\newdefinition{remark}{Remark}
\newdefinition{definition}{Definition}

\journal{Journal of Computational and Applied Mathematics}

\begin{document}

\begin{frontmatter}
	\title{An almost symmetric Strang splitting scheme for the construction of high order composition methods\tnoteref{label1}}
	\tnotetext[label1]{This work is supported by the Fonds zur F\"orderung der Wissenschaften (FWF) -- project id: P25346.}

	\author[uibk]{Lukas Einkemmer\corref{cor1}}
	\ead{lukas.einkemmer@uibk.ac.at}

	\author[uibk]{Alexander Ostermann}
	\ead{alexander.ostermann@uibk.ac.at}

	\address[uibk]{Department of Mathematics, University of Innsbruck, Austria}
	\cortext[cor1]{Corresponding author}

	\begin{abstract}
	In this paper we consider splitting methods for nonlinear ordinary differential equations in which one of the (partial) flows that results from the splitting procedure can not be computed exactly.
	{
{Instead, we insert a well-chosen state $y_{\star}$ into the corresponding nonlinearity $b(y)y$, which results in a linear term $b(y_{\star})y$ whose exact flow can be determined efficiently.}
Therefore, in the spirit of splitting methods, it is still possible for the numerical simulation to satisfy certain properties of the exact flow.}
 However, Strang splitting is no longer symmetric (even though it is still a second order method) and thus high order composition methods are not easily attainable. We will show that an iterated Strang splitting scheme can be constructed which yields a method that is symmetric up to a given order. This method can then be used to attain high order composition schemes. We will illustrate our theoretical results, up to order six, by conducting numerical experiments for a charged particle in an inhomogeneous electric field, a post-Newtonian computation in celestial mechanics, and a nonlinear population model and show that the methods constructed yield superior efficiency as compared to Strang splitting. For the first example we also perform a comparison with the standard fourth order Runge--Kutta methods and find significant gains in efficiency as well better conservation properties.
\end{abstract}

\end{frontmatter}

	\section{Introduction} \label{sec:introduction}
		If an ordinary differential equation can be cast in the form
		\[
			y' = A(y) + B(y),
		\]
		where the exact solutions of $y'=A(y)$, denoted by $\varphi_t^{A}(y(0))$, and $y'=B(y)$, denoted by $\varphi_t^{B}(y(0))$, are known, or can be computed efficiently, splitting methods often provide a viable alternative compared to more traditional integration schemes (such as Runge--Kutta methods). In addition, if the flows generated by $A$ and $B$ preserve a given property of the ordinary differential equation, so does the splitting scheme. In some instances this can be used to construct schemes which conserve certain properties of the exact flow (see, e.g.~\cite{hairer2006}). If the exact partial flows are used, the Strang splitting scheme with step size $\tau$, i.e.
		\[
			S_{\tau} = \varphi_{\frac{\tau}{2}}^A \circ \varphi_{\tau}^B \circ \varphi_{\frac{\tau}{2}}^A,
		\]
		is a symmetric scheme of second order. It is then possible to construct schemes of arbitrary (even) order by composition (see, e.g.~\cite{mclachlan1995}). For certain classes of ordinary differential equations more efficient schemes can be constructed (for the example of separable Hamiltonian systems see \cite{yoshida1990}). For a review of splitting methods we refer the reader to \cite{mclachlan:02}.

	However, even if one of the partial flows can not be computed exactly, in some circumstances splitting methods can still be applied. The systems of interest in this paper are ordinary differential equations which can be written as
	\begin{equation*} \label{eq:form}
		y' = A(y) + b(y)y + d,
	\end{equation*}
	where, as before, we assume that {$y'=A(y)$} can be solved exactly. However, no such assumption is made about {$y'=b(y)y+d$}. Instead, we assume that once a {fixed} value, say $y_{\star}$, is substituted, the flow corresponding to
	\[
		y' = b(y_{\star})y+d
	\]
	can be computed efficiently. We denote the corresponding flow by $\varphi_t^{b(y_{\star})}$, which can also be written explicitly by employing the exponential and $\phi_1$ function{s}. This yields
	\[
		\varphi_t^{b(y_{\star})}(y(0)) = \e^{t b(y_{\star})} y(0) + t \phi_1\left(t b(y_{\star})\right)d,
	\]
	where
	\[
		\phi_1(z)=\frac{\e^z-1}{z}.
	\]
	{Note that if we apply the Strang splitting scheme to $y'=A(y)+b(y_0)y+d$, a numerical methods results that is only of order $1$. This is intuitively clear, as $b$ is evaluated at the left endpoint only, and can be verified by a simple argument based on the Taylor expansion of the scheme. However, in the literature an alternative scheme has been proposed in the context of partial differential equations (see e.g.~\cite{cheng:1976}) that is usually referred to as Strang splitting also and is given by}
	\begin{subequations} \label{eq:standard_strang}
		\begin{equation}
		y_{1/2} = \varphi_{\frac{\tau}{2}}^{b(y_0)} \circ \varphi_{\frac{\tau}{2}}^A (y_0)
		\end{equation}
		\begin{equation}
		y_1 = M_{\tau} (y_0) = \varphi_{\frac{\tau}{2}}^A \circ \varphi_{\tau}^{b(y_{1/2})} \circ \varphi_{\frac{\tau}{2}}^A (y_0).
	\end{equation}
	\end{subequations}
	{Note that since we use an approximation of order $1$ to $b(y(\tau/2))$, this is in fact a method of order two. Consistent with the literature we will, from now on, refer to this scheme as Strang splitting.}
	
	For a symmetric scheme it must hold that (see, e.g.~\cite[Chap. II.3]{hairer2006})
	\[
	  M_{-\tau} \circ M_{\tau} = I,
	\]
	where $I$ denotes the identity. Now
	\[
		M_{-\tau} \circ M_{\tau}(y_0) =
		\varphi_{-\frac{\tau}{2}}^A \circ \varphi_{-\tau}^{b(\tilde{y}_{1/2})} \circ \varphi_{\tau}^{b(y_{1/2})} \circ \varphi_{\frac{\tau}{2}}^A (y_0),
	\]
	where
	\[
		\tilde{y}_{1/2} = \varphi_{-\frac{\tau}{2}}^{b(y_1)} \circ \varphi_{-\frac{\tau}{2}}^A(y_1).
	\]
	Inserting \eqref{eq:standard_strang} shows that
	\[
		\tilde{y}_{1/2} = \varphi_{-\frac{\tau}{2}}^{b(y_1)} \circ \varphi_{\tau}^{b(y_{1/2})} \circ \varphi_{\frac{\tau}{2}}^A (y_0).
	\]
	Therefore, the Strang splitting scheme is symmetric if and only if
	\[
		\varphi_{-\frac{\tau}{2}}^{b(y_1)} \circ \varphi_{\tau}^{b(y_{1/2})} =
		\varphi_{\frac{\tau}{2}}^{b(y_0)},
	\]
	which is not satisfied in general.
	
	Due to the lost symmetry, the corresponding triple jump scheme is only of order $3$ (not of order $4$, as one might naively expect). From this consideration it is also clear that further composition in the same manner does not result in schemes of arbitrary order (as is the case with the classical Strang splitting method {based on exact flows}).
	
	In section \ref{sec:almost_symmetric_strang}, we will propose a modified Strang splitting scheme that, in addition of being second order accurate, can be iterated to give a scheme that is symmetric up to a predetermined order $q$ {(as made precise in Definition~\ref{def:symmetric-order-y})}. Therefore, the usual construction of composition methods of arbitrary (even) order can be accomplished in this context (this is shown in section \ref{sec:composition}). In section \ref{sec:stiff} we show that for certain stiff problems the schemes constructed in this paper can be employed as well. In addition, we discuss in some detail the numerical results for three examples (namely for a charged particle in an inhomogeneous electric field, a post-Newtonian computation of celestial mechanics, and a nonlinear population model) in section \ref{sec:application}. Finally, we conclude in section \ref{sec:conclusion}.

  	\section{An almost symmetric Strang splitting scheme} \label{sec:almost_symmetric_strang}

	Let us start from the Lie splitting scheme
	\begin{equation} \label{eq:lie1}
		y_{1/2} = L_{\frac{\tau}{2}}(y_0) = \varphi_{\frac{\tau}{2}}^{b(y_0)} \circ \varphi_{\frac{\tau}{2}}^A( y_0).
	\end{equation}
	{We recall that the adjoint of a scheme $L_{\tau}$, which we denote by $L_{\tau}^{*}$, is defined as $L_{\tau}^{*}=L_{-\tau}^{-1}$. Therefore, to give a representation of the adjoint scheme corresponding to \eqref{eq:lie1} we interchange $y_{1/2}$ with $y_0$ and $\tau$ with $-\tau$. This yields}
	\begin{equation} \label{eq:fixed_point_eq}
		y_1 = \varphi_{\frac{\tau}{2}}^A \circ \varphi_{\frac{\tau}{2}}^{b(y_1)}( y_{1/2}),
	\end{equation}
	i.e. $y_1 = L^{*}_{\frac{\tau}{2}}(y_{1/2})$. Now the (implicit) Strang splitting scheme
	\[
		S_{\tau} = L^{*}_{\frac{\tau}{2}} \circ L_{\frac{\tau}{2}}
	\]
	is of second order and symmetric by construction (see, e.g. \cite{hairer2006}). However, it would require the solution of an implicit equation in each step, which is prohibitively expensive. Since equation \eqref{eq:fixed_point_eq} has the form of a fixed-point problem, we can employ fixed-point iteration to approximate $L^{*}_{\frac{\tau}{2}}$. We will denote the resulting scheme by $S^{(i)}_{\tau}$, where $i$ is the number of iterations that are conducted. {Note that during the iteration $y_{1/2}$ is fixed; that is, only the two evolution operators given explicitly in equation \eqref{eq:fixed_point_eq} are applied at each step in the fixed-point iteration.} As an initial value for the fixed-point iteration we employ $y_{1/2}$ (however, any approximation of order $\tau$ to $y_1$ would constitute a possible choice) and therefore {
	\[
		S^{(1)}_{\tau}(y_0)= \varphi_{\frac{\tau}{2}}^A \circ \varphi_{\frac{\tau}{2}}^{b(y_{1/2})} \big(\underset{y_{1/2}}{\underbrace{\varphi_{\frac{\tau}{2}}^{b(y_0)} \circ \varphi_{\frac{\tau}{2}}^A(y_0)}}\big)
	\]
}
	Clearly we need at least two iterations such that the scheme is of second order.
	There is no hope that $S^{(i)}_{\tau}$ is symmetric. However, we will show that it is almost symmetric as defined below.
	\begin{definition} \label{def:symmetric-order-y}
		A one-step method $\Phi_{\tau}$ is \emph{symmetric of order} $q$ if
		\begin{equation} \label{eq:symmetric_of_order_ass}
			\Phi^{*}_{\tau} = \Phi_{\tau} + \mathcal{O}\left( \tau^{q+1} \right),
		\end{equation}
		where $\Phi^{*}_{\tau}$ is the adjoint method of $\Phi_{\tau}$.
	\end{definition}
	
	Next let us show that the fixed-point iteration described above actually yields a scheme that is symmetric of order $i$.
	\begin{theorem} \label{thm:iteration} Suppose that $b(\cdot)$ is Lipschitz continuous. Then the Strang splitting scheme $S^{(i)}_{\tau}$ is symmetric of order $i$.
	\end{theorem}
	\begin{proof}
We have to show that the fixed-point problem~\eqref{eq:fixed_point_eq}, i.e., $y=F(y)$ with
\[
			F(y) =  \varphi_{\frac{\tau}{2}}^A \circ \varphi_{\frac{\tau}{2}}^{b(y)}(y_{1/2})
\]
has a unique solution in a sufficiently small neighborhood of $y_{1/2}$. First note that there exists a constant $C>0$ such that
\[
\| F(y_{1/2})-y_{1/2}\| \le C\tau
\]
for $\tau$ sufficiently small. Now, let $D=2C$ and denote by $\Omega_D$ the closed ball with center $y_{1/2}$ and radius $D\tau$. Then, for all $u, v \in \Omega_D$ it holds
\begin{equation}\label{eq:lip-bound1}
\|F(u)-F(v)\| \le \left\| \varphi_{\frac{\tau}{2}}^A\right\| \cdot \left\|\varphi_{\frac{\tau}{2}}^{b(u)}(y_{1/2})	- \varphi_{\frac{\tau}{2}}^{b(v)}(y_{1/2})\right\|,
\end{equation}
where $\Vert \varphi_{\frac{\tau}{2}}^A \Vert = 1+ \mathcal O(\tau)$ denotes the Lipschitz constant of $\varphi_{\frac{\tau}{2}}^A (\cdot)$ on the bounded set
\[
\Omega = \bigcup_{y\in \Omega_D} \left\{ \varphi_{\frac{\tau}{2}}^{b(y)}(y_{1/2})\right\}.
\]
Further note that
\begin{equation}\label{eq:lip-bound2}
\left\|\varphi_{\frac{\tau}{2}}^{b(u)}(y_{1/2})	- \varphi_{\frac{\tau}{2}}^{b(v)}(y_{1/2})\right\| \le \frac{\tau}2 \bigl(1+\mathcal O(\tau)\bigr) \|b(u)-b(v)\| \|y_{1/2}\|,
\end{equation}
which is a direct consequence of the variation-of-constants formula. By combining the bounds~\eqref{eq:lip-bound1}, \eqref{eq:lip-bound2} with the Lipschitz continuity of $b(\cdot)$, we obtain that $F$ is Lipschitz continuous on $\Omega_D$ with a Lipschitz constant of order $\tau$. Moreover, using the triangle inequality we get the bound
\[
\|F(y) - y_{1/2}\| \le \|F(y)-F(y_{1/2})\| + \|F(y_{1/2}) - y_{1/2}\| \le D\tau
\]
for all $y\in \Omega_D$ and $\tau$ sufficiently small. This shows that $F$ maps the closed ball $\Omega_D$ onto itself. Consequently, by Banach's fixed-point theorem, $F$ has a unique fixed-point $y_1$ in $\Omega_D$, which is the locally unique solution of~\eqref{eq:fixed_point_eq}.

Since $S^{(1)}_{\tau}(y_0) = F(y_{1/2})$ and $S_{\tau}(y_0)=F(y_1)$, we also obtain that
\[
\left\|S^{(1)}_{\tau}(y_0) - S_{\tau}(y_0)\right\| = \|F(y_{1/2}) - F(y_1)\| \le L\tau \|y_{1/2}-y_1\| \le LD\tau^2.
\]
Moreover, as the Lipschitz constant of $F$ is of order $\tau$ this implies
\begin{equation}\label{eq:last}
		S^{(i)}_{\tau}(y_0) = S_{\tau}(y_0) + \mathcal{O}(\tau^{i+1}).
\end{equation}
Also recall that $S_{\tau}$ is a symmetric scheme by construction. Thus, \eqref{eq:last} proves the desired result.
	\end{proof}
	
	Thus, we have established that we can iteratively compute a second order method that is symmetric to arbitrary order. Moreover, the computational effort is linear in the desired order of symmetry.
	
	In the next section we will discuss how the scheme described here can be used to construct composition methods of arbitrary (even) order.
	
	\section{Composition methods} \label{sec:composition}
	
	It is well-known (see, e.g.~\cite[Chap. II.4]{hairer2006}) that if a symmetric one-step method $\Phi_{\tau}$ of even order {$r$} is composed in the following manner
	\begin{equation} \label{eq:composition1}
		\Phi_{\gamma_3 \tau} \circ \Phi_{\gamma_2 \tau} \circ \Phi_{\gamma_1 \tau},
	\end{equation}
	where
	\begin{equation}\label{eq:composition2}
		\gamma_1 = \gamma_3 = \frac{1}{2-2^{1/({r}+1)}}, \quad
		\qquad \gamma_2 = -2^{1/({r}+1)} \gamma_1,
	\end{equation}
	then a one-step method of order ${r}+2$ results. Thus, we can construct methods of arbitrary even order {$p$}, where the cost, in terms of a single evaluation of the corresponding second order method, is given by $3^{p/2-1}$. For $p=4$, for example, the corresponding method is the well-known triple jump scheme.
	
	The justification for this procedure is given by Theorem 4.1 in \cite{hairer2006}. We will now generalize that result for methods that are (only) symmetric of order $q$.
	\begin{lemma} \label{thm:generalization_almost_symmetric}
		Suppose that the one-step method $\Phi_{\tau}$ is of \emph{odd} order $p$ and symmetric of order $q$ with $q\geq p+1$, {see \eqref{eq:symmetric_of_order_ass}.} Then, the method is in fact of order $p+1$.
	\end{lemma}
	\begin{proof}
		Let us denote the exact flow by $\varphi_\tau$. Since the method has order $p$,
		\[
			\Phi_{\tau}(y_{0})-\varphi_{\tau}(y_{0})=C(y_{0})\tau^{p+1}+\mathcal{O}\left(\tau^{p+2}\right)
 		\]
		and further the adjoint method satisfies
		\[
			 \Phi_{\tau}^{*}(y_{0})-\varphi_{\tau}(y_{0})=(-1)^{p}C(y_{0})\tau^{p+1}+\mathcal{O}\left(\tau^{p+2}\right).
 		\]
		Using now assumption \eqref{eq:symmetric_of_order_ass}{, with $p$ odd,} we get
		\[
			 (-1)^{p}C(y_{0})\tau^{p+1}=C(y_{0})\tau^{p+1}+\mathcal{O}(\tau^{q+1})+\mathcal{O}\left(\tau^{p+2}\right)
 		\]
		and thus $C(y_{0})=0$ if $q\geq p+1$. Therefore, we deduce that $\Phi_{\tau}$ is of order $p+1$.
	\end{proof}
	
	As a corollary we get the desired order for the composition methods as well as the number of iterations we have to perform. As we will see in section \ref{sec:chargedparticle} this is a worst case estimate that can be improved upon for some applications.
	
	\begin{corollary} The composition method constructed from $S^{(i)}_{\tau}$ by using $\ell$ compositions, as described in equation \eqref{eq:composition1}, results in a scheme of order $p = 2+2\ell$ if $i \geq p$.
	\end{corollary}
	\begin{proof}
		If $i\geq 2$ then $S^{(i)}_{\tau}$ is of order $2$ by construction. Thus, the composition method is at least of order $3$. However, from Lemma \ref{thm:generalization_almost_symmetric} we know that if $i\geq 4$ this method is in fact of order $4$. Since the composition given in \eqref{eq:composition1} is symmetric, a method which is symmetric of order $q$ retains this property if composed in the manner described. Therefore, we can complete the proof by induction.
	\end{proof}
  	
		Before we turn our attention to the applications given in the next section, let us investigate the (worst case) computational cost of the composition methods considered in this section. In Table \ref{tbl:effort} the number of computations of either $\varphi_{\tau}^A$ or $\varphi_{\tau}^{b(q_{\star})}$ is given for the triple jump scheme as well as the composition of the triple jump scheme (which we call composite 9). In addition, Table \ref{tbl:effort} lists an abbreviation of all the schemes discussed (which we will employ heavily in the next section). The methods constructed here will be referred to as \emph{iterated}.
	
	\begin{table}[H]
		\caption{The effort in number of (possibly approximated) partial flows that have to be computed is listed for a number of composition schemes. In addition, the abbreviations used for the composition methods employed in the next section are given. \label{tbl:effort}}
		\vspace{0.2cm}
		\centering
		{\small
		\begin{tabular}{ l l l l l }
			Method & Abbreviation & Order & Iterations & Effort \\
			\hline
  			Strang \eqref{eq:standard_strang} & S & 2 & - & 4 \\
  			Iterated Strang & IS		& 2 & 2 & 6 \\
			\hline
			Triple jump & TJ				& 3 & - & 12 \\
  			Iterated triple jump & ITJ		& 4 & 4 & 30 \\  	
  			\hline
  			Composite 9 & C9				& 3 & - & 36 \\
  			Iterated composite 9 & IC9		& 6 & 6 & 126  \\  	  					
		\end{tabular}
		}
	\end{table}  	
  	
	Note that even though the high order methods given in Table \ref{tbl:effort} are about three times as costly as conventional composition methods (which are employed for separable Hamiltonian systems, for example), we are now able to construct methods of arbitrary (even) order. We will show in section~\ref{sec:application} that for realistic problems this can still result in a considerable gain in performance (as compared to the more commonly employed Strang splitting scheme, for example).
	
	In addition, it should be duly noted that similar to conventional composition methods, the schemes introduced here conserve all invariants that are invariants of the two partial flows as well.	
	To conclude this section, let us remark that the schemes introduced here do not require any modification in the code used to implement the numerical solution of the partial flows. That is, if for a given problem the Lie or Strang splitting scheme is already implemented, the generalization to the methods discussed here is almost immediate.
  	
	\section{Extension to stiff problems \label{sec:stiff}}
	
	In this section we show that in certain circumstances we can extend our analysis to the stiff case. Let us consider, for example, an ordinary differential equation for which the {operator $b$}, as defined in section \ref{sec:introduction}, can be written as{
	\begin{equation} \label{eq:stiff_nonstiff}
		b(y) = b_S + b_N(y),
	\end{equation}}%
i.e., the nonlinear operator $b$ can be split in a stiff linear part and a non-stiff nonlinear part. In this case we can show that the speed of convergence of the fixed-point iteration in  Theorem \ref{thm:iteration} (see section \ref{sec:almost_symmetric_strang}) is independent of the stiff part. This is the content of the following corollary.
	\begin{corollary} Suppose that $b(\cdot)$ can be cast into the form \eqref{eq:stiff_nonstiff} with $b_N(\cdot)$ Lipschitz continuous. Then the Strang splitting scheme $S_{\tau}^{(i)}$ is symmetric of order $i$ and the error of the method can be estimated independently of $\Vert b_S \Vert$.
	\end{corollary}
	\begin{proof} We employ the variation-of-constants formula to get
		\[
			\varphi_{\frac{\tau}{2}}^{b(u)}(y_{1/2}) - \varphi_{\frac{\tau}{2}}^{b(v)}(y_{1/2}) = \int_0^{\frac{\tau}{2}} {\varphi_{\frac{\tau}{2}-\sigma}^{b_S} \left( b_N(u)y_{1/2} - b_N(v)y_{1/2} \right)} \,\mathrm{d}\sigma
		\]
		which allows us to estimate
		\[
			\left\Vert \varphi_{\frac{\tau}{2}}^A \circ \varphi_{\frac{\tau}{2}}^{b(u)}(y_{1/2}) - \varphi_{\frac{\tau}{2}}^A \circ \varphi_{\frac{\tau}{2}}^{b(v)}(y_{1/2})	\right\Vert \leq C \tau \Vert b_N(u)-b_N(v) \Vert \Vert y_{1/2} \Vert,
		\]
		where $C$ depends on $\Vert \varphi_{\frac{\tau}{2}}^A \Vert$ and $\Vert \varphi_{\frac{\tau}{2}}^{b_S}	\Vert$ but not on $\Vert b_S \Vert$.

The proof is completed by employing the same arguments used in the proof of Theorem~\ref{thm:iteration}.
	\end{proof}
	Therefore, we have shown that, in the situation described, the step size can be chosen independently of the stiff part of the problem. This is of interest in some applications, where the inclusion of $b_S$ in the operator $A$ would result in partial flows that are more difficult to compute, or where a conservation property of the problem under consideration is destroyed if $b_S$ is treated separately from the nonlinearity $b_N$.

	To conclude this section let us briefly discuss the Brusselator (which is described in \cite[Chap. IV.1]{hairer2}). In this case we have a discretized diffusion-reaction equation, where the flow corresponding to $A$ can be computed very efficiently by employing fast Fourier transform techniques. The remaining stiffness in the system is then only due to the linear part of the flow corresponding to $b$. In addition, an analytical expression of the partial flow corresponding to the nonlinearity is not easily attainable (due to the coupling of the equations involved). Therefore, the problem is of the form considered in this section. The implementation and analysis of such methods in the context of partial differential equations is the subject of further research.

	\section{Applications} \label{sec:application}
  	
	In this section, we discuss three applications of the schemes constructed in the previous sections. First, we consider a Hamiltonian system that describes the movement of a charged particle in an inhomogeneous electromagnetic field. This system will turn out to require fewer iterations for a desired order of symmetry as compared to the worst case described in section \ref{sec:composition}. Second, we consider a post-Newtonian approximation to the relativistic Kepler problem. Also in this case we will observe that fewer iterations are necessary, compared to the worst case, to construct schemes of order four and six. Third, a nonlinear population model is considered. This model, in fact, exhibits the worst case behavior as outlined in section \ref{sec:composition}. Nevertheless, we can show, by conducting numerical experiments, that in all three examples the use of high order methods results in a significant performance increase.

	{
	In all the simulations conducted, we compute a reference solution by using the (classic) Strang splitting scheme and a sufficiently small (experimentally determined) step size.
	}
  	
  	\subsection{A charged particle in an inhomogeneous magnetic field} \label{sec:chargedparticle}
  	
	The equations of motion of a charged particle in an external electromagnetic field are given by the Lorentz force law
	\[
		m \ddot{x} = q(E+v \times B),
	\]
	where $x$, $v$, $q$ are the particle's position, velocity, and charge, respectively; the electric field is denoted by $E$ and the magnetic field by $B$ (both can depend on the position of the particle under consideration, i.e.~on $x$). This differential equation can be reformulated as a Hamiltonian system with Hamiltonian
	\[
		H = \frac{p^2}{2m} + q \phi,
	\]
	where the electric potential $\phi$ is related to the electric field by $E = -\nabla \phi$. We should note that the momentum $p=m v$ used above is not the conjugate variable to the position (as would be the case in the electrostatic limit).
	
	The equations of motion in this framework are then given by 	
	\begin{align*}
		\dot{x} &= p/m \\
		\dot{p} &= F(x) + \Omega(x)p,
	\end{align*}
	where
	\[	
		F = q E, \qquad
		\Omega = \left[\begin{array}{ccc}
0 & \tilde{B}_{3} & -\tilde{B}_{2}\\
-\tilde{B}_{3} & 0 & \tilde{B}_{1}\\
\tilde{B}_{2} & -\tilde{B}_{1} & 0
\end{array}\right]
	\]	
	with $\tilde{B}_i = q B_i / m$. To set up the splitting, we use
	\[
		A(x,p) = \left[\begin{array}{c}
0\\
F(x)
\end{array}\right], \qquad
{
	b(x_{\star},p_{\star}) = \left[\begin{array}{cc} 0& \frac1m I \\[1mm] 0 & \Omega(x_{\star})
\end{array}\right],\qquad d=0}
	\]
	and therefore
	\[
		\varphi_{\tau}^A(x_0,p_0) =\left[\begin{array}{c}
x_{0}\\
p_{0}+\tau F(x_{0})
\end{array}\right]
	\]
	whereas the second partial flow can be computed exactly once we substitute $x_{\star}$ (and thus consider $\Omega$ to be constant). The analytic expression is given by
	\begin{equation} \label{eq:exact_B}
		\varphi_{\tau }^{b(x_{\star})}(x_0,p_0) =\left[\begin{array}{c}
				 \frac{1}{m} \int_0^{\tau}\exp\left({s\Omega(x_{\star})}\right) p_{0}\,\mathrm{d}s + x_{0}\\
\exp\left({\tau\Omega(x_{\star})}\right) p_{0}
\end{array}\right].
	\end{equation}
	For actual computations we can use  	
	\[
		 \exp(\tau\Omega)=I+\frac{\sin\tau\Vert\tilde{B}\Vert_{2}}{\Vert\tilde{B}\Vert_{2}}\Omega+\frac{1-\cos\tau\Vert\tilde{B}\Vert_{2}}{\Vert\tilde{B}\Vert_{2}^{2}}\Omega^{2}
	\]  	
	and
	\[
		\int_0^{\tau} \exp(s\Omega) \,\mathrm{d}s =
		\tau I+\frac{1-\cos\tau\Vert\tilde{B}\Vert_{2}}{\Vert\tilde{B}\Vert_{2}^{2}}\Omega+\frac{\tau\Vert\tilde{B}\Vert_{{{2}}}-\sin\tau\Vert\tilde{B}\Vert_{2}}{\Vert\tilde{B}\Vert_{2}^{3}}\Omega^{2},
	\]
{where both $\Omega$ and $\tilde{B}$ depend on $x_{\star}$; this dependence is, for the sake of brevity, omitted in the notation used}. Thus, we have fulfilled all the requirements outlined in section \ref{sec:introduction}. Note that for a uniform magnetic field a number of symmetric second order schemes are available (see, e.g.~\cite{spreiter1999}). However, for non-uniform magnetic fields such schemes can not be employed to get higher order schemes by composition.
	
	Let us now discuss a peculiarity of the system under consideration. As $B(\cdot)$ does only depend on the position component of the phase space and the evolution operator $\varphi_{\tau}^{A}$ does not depend on the momentum (which is a consequence of the specific splitting conducted here), we have
	\[
		\int_0^{\tau} \exp\left({s\Omega(x_{2})}\right) p_{0}\,\mathrm{d}s
		- \int_0^{\tau} \exp\left({s\Omega(x_{1})}\right) p_{0}\,\mathrm{d}s
		= \mathcal{O}\left( \tau^2 \Vert \Omega(x_2) - \Omega(x_1) \Vert \right).
	\]
	That is, the Lipschitz constant for our fixed-point iteration is of order $\tau^2$. However, such a result is not entirely unexpected as it is quite common that the position is integrated in time with a higher order than the momentum component (this is also true for the popular leapfrog scheme, for example). Therefore, any resulting approximation to $y_{1}=(x_1,p_1)$ is of order $2\ell+1$, for some \mbox{$\ell\in\mathbb{N}$}, in position and, as we can easily deduce from equation \eqref{eq:exact_B}, the momentum is then approximated up to order $2\ell$. Therefore, to use the notation from section \ref{sec:almost_symmetric_strang} we have a symmetric scheme of order $q=2\ell-1$.
	
	Thus, three iterations are sufficient to get a fourth order scheme whereas four iterations suffice to get a sixth order scheme. This is clearly below the worst case behavior discussed in section \ref{sec:composition}.
	
	We now turn our attention to the presentation of the numerical simulations conducted. As an example we will use an electric field configuration that corresponds to an ideal Penning trap (such as described in \cite{kretzschmar2000}). However, we will use a magnetic field that is not homogeneous in space. {Further we} will use natural units for the problem, i.e., $m$ and $q$ are set to unity. For the ideal Penning trap the electric potential is given by
	\[
		\phi(x) = \frac{1}{20}\left( 2 x_3^2 - x_1^2 - x_2^2 \right).
	\]
	In order to impose an inhomogeneous magnetic field, we use
	\[
		B(x) = \left[ \tfrac{1}{10}x_3, \tfrac{1}{10} x_2, 100 \sin x_3+x_2 \right]^{\sf{T}}.
	\]
	We consider an initial value in both position as well as momentum close to zero and evolve the system until time $T=100$. In Figure \ref{fig:chargedparticle_error_vs_tau} we show that the numerical experiments match the expected order for the splitting schemes discussed in section \ref{sec:composition}. Note, however, that for the composite 9 scheme only three iterations are required to reach order six (instead of the four predicted above). This is a clear indication at the presence of further simplifications (in the system under consideration).
	Now let us turn our attention to run time considerations. In Figure \ref{fig:chargedparticle_time_vs_error} the run time is plotted against the achieved accuracy.
	It is clear from that figure that even for moderate precision requirements, high order methods provide a significant advantage over the more commonly employed Strang splitting scheme.
	
	\begin{figure}[t]
		\centering
		\resizebox{0.73\textwidth}{!}{\small \input{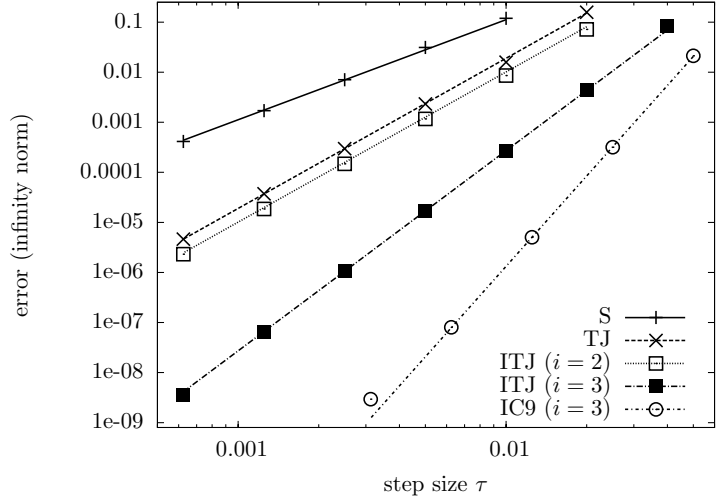}}
		\caption{Order plot for a charged particle in an inhomogeneous magnetic field (the results for various splitting schemes are shown). The lines drawn are, from top to bottom, of slope $2$, $3$, $3$, $4$, and $6$ respectively. The abbreviations for the different numerical schemes are listed in Table \ref{tbl:effort}. \label{fig:chargedparticle_error_vs_tau}}
	\end{figure}
  	\begin{figure}[t]
		\centering
		\resizebox{0.74\textwidth}{!}{\small \input{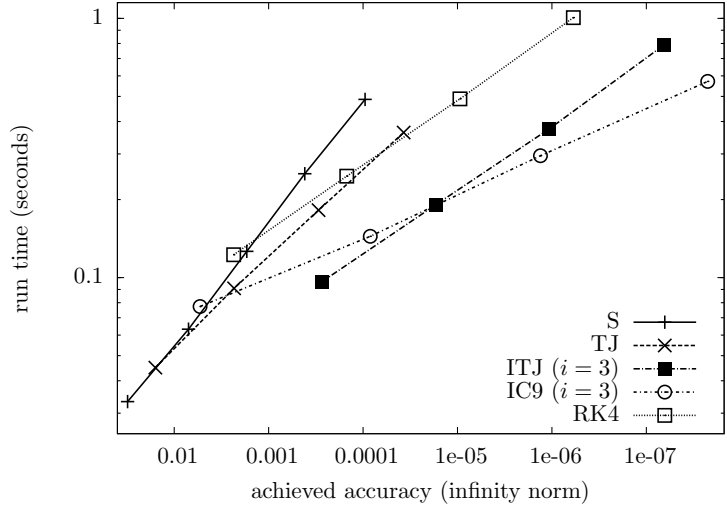}}
		\caption{Run time as a function of the achieved accuracy for a charged particle in an inhomogeneous magnetic field (the results for various splitting schemes are shown). The abbreviations for the different numerical schemes are listed in Table \ref{tbl:effort}. For comparison, the standard Runge--Kutta scheme of order four (RK4) is also shown.  \label{fig:chargedparticle_time_vs_error}}
	\end{figure}  	
	The system under consideration is Hamiltonian; therefore, the energy is exactly conserved. This is, in general, no longer true if a numerical scheme is considered. However, schemes can be engineered which, to machine precision, conserve the energy. It is clear that this is not true in this case as the partial flows do not conserve the energy (Figure \ref{fig:chargedparticle_conservation} confirms this behavior). However, the error in energy is still four orders of magnitude below the integration error made by the scheme under consideration.

	\begin{figure}[H]
		\centering
		\resizebox{0.73\textwidth}{!}{\small \input{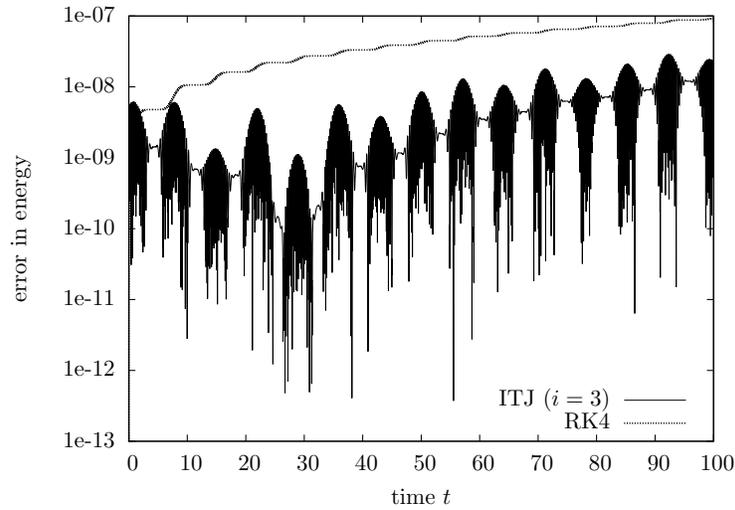}}
		\caption{Energy conservation for a charged particle in an inhomogeneous magnetic field, where the iterated triple jump scheme ($i=3$) with $\tau=0.01$ is employed. This results in an error (in the infinity norm) in the position/momentum that is approximately $3\cdot 10^{-4}$. For comparison, the standard Runge--Kutta method of order four is shown. There the step size $\tau=0.0015$ is chosen, which results in a comparable accuracy and twice the run time. Note, however, that the error in energy is better by an order of magnitude for the iterated triple jump scheme. \label{fig:chargedparticle_conservation}}
	\end{figure}
  	
	To end this section let us note that the discussion here can easily be generalized to multiple particles. This is still true if particle-particle interactions (via the electric or magnetic field, for example) are considered.  	 
  	
	\subsection{Post-Newtonian Kepler problem} \label{sec:postnewtonian}
	As a second example we consider the post-Newtonian\footnote{The equations of motion in the post-Newtonian approximation are determined by expanding the field equations of general relativity for point objects in powers of $1/c^2$.} approximation to the (general) relativistic $n$-body problem. In this section we will limit ourselves to the relativistic Kepler problem in the Post-Newtonian approximation up to terms of order $1/c^4$, where $c$ denotes the speed of light. The equations of motions for the first body are then given by (see, e.g.~\cite{blanchet:2001})
	\begin{align*}
		\dot{r}_1 &= v_1  \\
		\dot{v}_1 &=
		-\frac{\mu_2}{r_{12}^2}n_{12} + \frac{1}{c^2}\left( 5 \frac{\mu_1 \mu_2}{r_{12}^3} + 4\frac{\mu_2^2}{r_{12}^3} \right) n_{12} \\
		&\qquad + \frac{1}{c^2}\frac{\mu_2}{r_{12}^2} \left( \frac{3}{2} (n_{12}\cdot v_2)^2 - v_1^2 + 4v_1\cdot v_2 -2 v_2^2 \right) n_{12} \\
		&\qquad + \frac{1}{c^2}\frac{\mu_2}{r_{12}^2} \left( 4 n_{12}\cdot v_1 - 3n_{12} \cdot v_2 \right) (v_1 - v_2),
	\end{align*}
	where $r_{12} = \Vert r_1 - r_2 \Vert_{2}$, $n_{12} = \tfrac{1}{r_{12}} (r_1 - r_2)$, and $\mu_i = G m_i$ is the standard gravitational parameter (which can be computed from the gravitational constant $G$ and the mass of the body $m_i$). The equations of motion for the second body can then be determined by interchanging the indices corresponding to the first and the second body in the equations of motion stated above. Let us note that the Newtonian equations of motion are recovered in the limit as $c\to\infty$ (in this case only the first force term remains).
	The structure of the equations of motion naturally lends itself to the splitting scheme described in section \ref{sec:introduction}. To that end let us define
	\[
		A(r_1,v_1,r_2,v_2) =
	\left[\begin{array}{c}
					0\\
					-\frac{\mu_{2}}{r_{12}^{2}}n_{12}+
					 \frac{1}{c^2}\left(5\frac{\mu_{1}\mu_{2}}{r_{12}^{3}}+4\frac{\mu_{2}^{2}}{r_{12}^{3}}\right)n_{12}\\
					0\\
					\frac{\mu_{1}}{r_{12}^{2}}n_{12}-\frac{1}{c^2}\left( 5\frac{\mu_{1}\mu_{2}}{r_{12}^{3}}+4\frac{\mu_{1}^{2}}{r_{12}^{3}}\right)n_{12}
				\end{array}\right]
	\]
	and
	\[
		b(r_{1\star},v_{1\star},r_{2\star},v_{2\star})=\left[\begin{array}{cccc}
				0 & 1 & 0 & 0\\
				K_{1} & L_{1} & -K_{1} & -L_{1}\\
				0 & 0 & 0 & 1\\
				K_{2} & L_{2} & -K_{2} & -L_{2}
			\end{array}\right],
	\]
	where
	\begin{align*}
		K_1 &= \frac{1}{c^2}\frac{\mu_2}{r_{12\star}^3} \left( \frac{3}{2} (n_{12\star}\cdot v_{2\star})^2 - v_{1\star}^2 + 4v_{1\star}\cdot v_{2\star} -2 v_{2\star}^2 \right), \\
		L_1 &= \frac{1}{c^2}\frac{\mu_2}{r_{12\star}^2} \left( 4 n_{12\star}\cdot v_{1\star} - 3n_{12\star} \cdot v_{2\star} \right).
	\end{align*}
	The corresponding quantities $K_2$ and $L_2$ can once again be obtained by reversing the indices corresponding to the first and second body. It is clear that the flows corresponding to both $A$ and $B(r_{1\star},v_{1\star},r_{2\star},v_{2\star})$, as defined above, can be computed efficiently.
	
	In the subsequent discussion, we will employ the SI system of units (for convenience we will not state the units explicitly). Let us consider the orbit of two celestial objects with $\mu_1 = 10^{26}$, i.e., approximately $0.75\cdot 10^6$ solar masses, and $\mu_2=10^{20}$. We initialize the first body with zero velocity and the second one with $v_2 = 5.898\cdot 10^6$ and place it at the perihelion of the orbit which we determine to be $r_2 = 4.6\cdot 10^{10}$, i.e., a mercury like orbit. We integrate the equations of motion up to the final time $T=10^6$, which corresponds to about fifteen orbits. The order plots for a number of schemes are shown in Figure \ref{fig:postnewtonian_error_vs_tau}.

	The number of iterations necessary for the iterated triple jump scheme (ITJ) as well as the iterated composite 9 scheme (IC9) have been determined by conducting numerical experiments. A theoretical analysis is beyond the scope of this paper. Note, however, that similar to the previous example we do not observe the worst case behavior described in section \ref{sec:composition}.

In addition, let us investigate the run time as a function of the error. This is shown in Figure \ref{fig:postnewtonian_time_vs_error}.
As is apparent from the figure, the fourth order iterated triple jump scheme (ITJ) is superior to both the third order triple jump scheme and the Strang splitting scheme.
For medium accuracy requirement it becomes advantageous to employ the sixth order IC9 scheme.

		\begin{figure}[H]
			\centering
			\resizebox{0.73\textwidth}{!}{\small \input{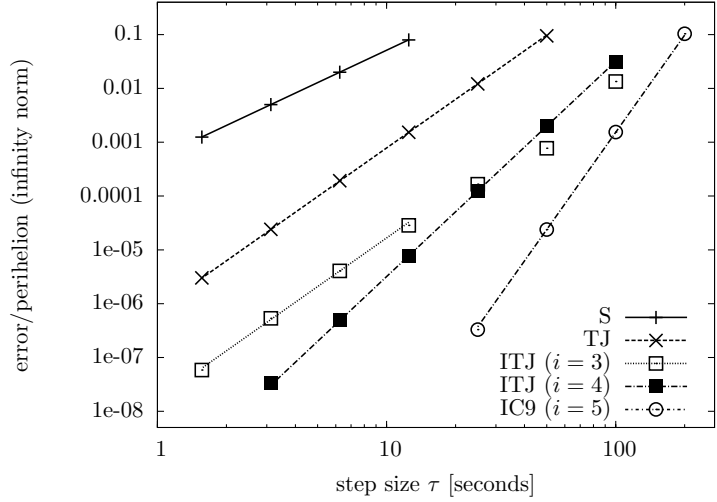}}
			\caption{ Order plots for a post-Newtonian Kepler problem (the results for various splitting schemes are shown). The lines drawn are, from top to bottom, of slope 2, 3, 3, 4, and 6 respectively. The error is scaled to the perihelion (the point of least distance between the two bodies) of the orbit. The abbreviations for the different numerical schemes are listed in Table \ref{tbl:effort}.  \label{fig:postnewtonian_error_vs_tau}}
		\end{figure}
			\begin{figure}[H]
			\centering
			\resizebox{0.73\textwidth}{!}{\small \input{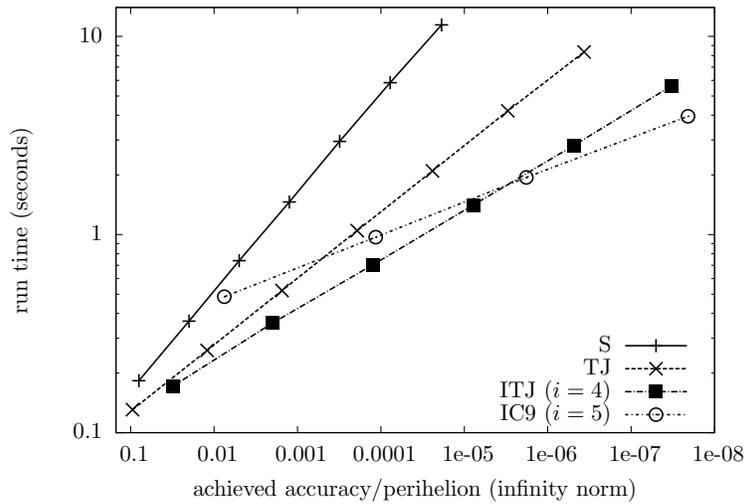}}
			\caption{ Run time as a function of the achieved accuracy for a post-Newtonian Kepler problem (the results for various splitting schemes are shown). The error is scaled to the perihelion of the orbit. The abbreviations for the different numerical schemes are listed in Table \ref{tbl:effort}. \label{fig:postnewtonian_time_vs_error}}
		\end{figure}

	\subsection{A nonlinear population model} \label{sec:nonlinerpopulation}

	As a third example, we consider a nonlinear population model (the so called May model) that is given by
	\begin{align*}
		x' &= a x \left( 1 - \frac{x}{b} \right) - \frac{c x y}{x+d} \\
		y' &= e y - \frac{y^2}{f x},
	\end{align*}
	where in line with \cite{callahan93} we use $a=0.6$, $b=10.0$, $c=0.5$, $d=1.0$, $e=0.1$, and $f=2.0$. In this context, $x$ is interpreted as a (appropriately scaled) prey population while $y$ represents the predator population. We can argue that such an equation lends itself to splitting as if interaction effects are neglected we are usually left with either an exponential growth model or a logistic equation in each variable. In fact, this is the case for the equation stated above, since the decoupled system can be written as
	\[
		\left[\begin{array}{c}
x'\\
y'
\end{array}\right]=
{A(x,y)}
=\left[\begin{array}{c}
ax\left(1-\frac{x}{b}\right)\\
ey
\end{array}\right],
	\]
	of which an analytical solution can easily be found; it is given by
	\begin{align*}
		x(t) &= \frac{b \e^{a t}}{\e^{at} - 1 + \frac{b}{x(0)}} \\
		y(t) &= \e^{ e t} y(0).	
	\end{align*}
  	To complete our splitting scheme, we set \[
b\left(x_{\star},y_{\star}\right)\left[\begin{array}{c}
x\\
y
\end{array}\right]=\left[\begin{array}{c}
-\frac{cy_{\star}}{x_{\star}+d}x\\
-\frac{y_{\star}}{fx_{\star}}y
\end{array}\right]
  	\]
	which once again is exactly the situation described in section \ref{sec:introduction}.

	One might rightfully object that our splitting approach is somewhat artificial as we can simply add the $A$ operator to the $B$ operator. After all, the resulting operator still has the desired form and splitting would not be necessary. The only potential advantage of using the splitting scheme is that we can solve the flow corresponding to $A$ exactly. Although numerical experiments demonstrates that this can result in a significant increase in performance, the goal of this section is to show that the number of iterations given in section \ref{sec:composition} constitutes a sharp bound.
	
	To investigate that behavior let us choose the initial values $x(0)=100$ and $y(0)=20$, i.e., the prey population is significantly larger than the predator population. In Figure~\ref{fig:maymodel_time_vs_error} we plot the run time as a function of the error for a number of schemes discussed so far (we integrate up to $t=5$).

	It is also clear that contrary to the example discussed in the previous section, high order schemes (beyond triple jump) are only advantageous if very high precision is needed; however, this behavior is not surprising as the solution approaches a steady state quite rapidly.

	Therefore, let us now turn our attention to the number of iterations necessary to obtain a given order. In Figure \ref{fig:maymodel_error_vs_tau} we can clearly see that the behavior described in section \ref{sec:composition} is regained. Thus, the system under consideration does not possess the simplifying property we discussed in section \ref{sec:chargedparticle} for the charged particle and in section \ref{sec:postnewtonian} for the post-Newtonian approximation. We can also conclude that the number of iterations given in section \ref{sec:composition} constitutes a sharp bound.
 	
	\begin{figure}[H]
		\centering
		\resizebox{0.73\textwidth}{!}{\small \input{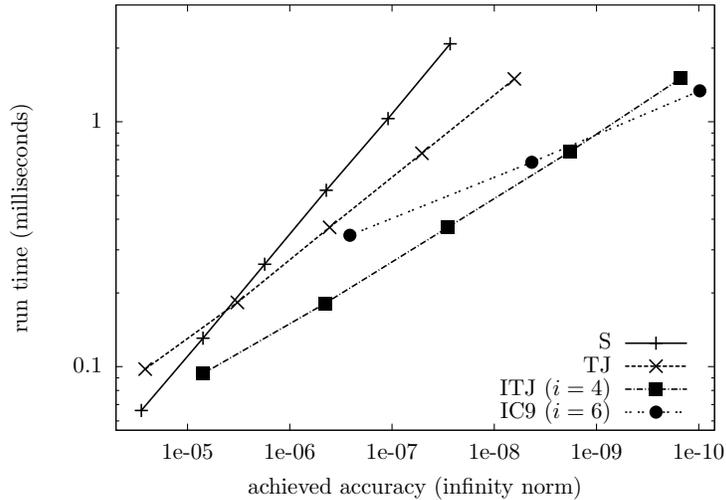}}
		\caption{ Run time as a function of the achieved accuracy for the May model. The abbreviations for the different numerical schemes are listed in Table \ref{tbl:effort}. \label{fig:maymodel_time_vs_error}}
	\end{figure}

	\begin{figure}[H]
		\centering
		\resizebox{0.73\textwidth}{!}{\small \input{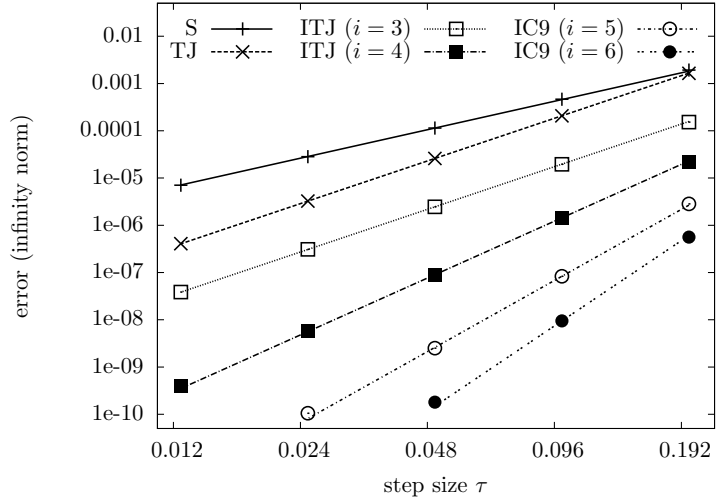}}
		\caption{Order plot for the May model (the results for various splitting schemes are shown). The lines drawn are, from top to bottom and left to right, of slope $2$, $3$, $3$, $4$, $5$, and $6$ respectively. The abbreviations for the different numerical schemes are listed in Table \ref{tbl:effort}. \label{fig:maymodel_error_vs_tau}}
	\end{figure}

   As in the previous section we note that generalizations, for example the inclusion of multiple predator species, can be easily accomplished in the context of the schemes discussed.

   	\section{Conclusion} \label{sec:conclusion}
  	
	Besides providing a theoretical analysis, we have conducted numerical simulations that demonstrate the applicability of {composition} schemes to three examples of interest in the sciences. In all of these examples we have demonstrated that, depending on the accuracy requirement, the high order schemes constructed in this paper can provide significant gains in performance compared to Strang splitting. For a charged particle in an inhomogeneous field, we have also demonstrated increased efficiency as well as better conservation properties as compared to the standard fourth order Runge--Kutta method.

  	\bibliography{paper_ode}

\begin{thebibliography}{10}
\expandafter\ifx\csname natexlab\endcsname\relax\def\natexlab#1{#1}\fi
\providecommand{\bibinfo}[2]{#2}
\ifx\xfnm\relax \def\xfnm[#1]{\unskip,\space#1}\fi
\bibitem[{Hairer et~al.(2006)Hairer, Lubich, and Wanner}]{hairer2006}
\bibinfo{author}{E.~Hairer}, \bibinfo{author}{C.~Lubich},
  \bibinfo{author}{G.~Wanner}, \bibinfo{title}{{Geometric Numerical
  Integration: Structure-Preserving Algorithms for Ordinary Differential
  Equations}}, \bibinfo{publisher}{Springer-Verlag, Berlin Heidelberg},
  \bibinfo{year}{2006}.
\bibitem[{McLachlan(1995)}]{mclachlan1995}
\bibinfo{author}{R.~McLachlan},
\newblock \bibinfo{title}{On the numerical integration of ordinary differential
  equations by symmetric composition methods},
\newblock \bibinfo{journal}{SIAM J. Sci. Comput.} \bibinfo{volume}{16}
  (\bibinfo{year}{1995}) \bibinfo{pages}{151--168}.
\bibitem[{Yoshida(1990)}]{yoshida1990}
\bibinfo{author}{H.~Yoshida},
\newblock \bibinfo{title}{Construction of higher order symplectic integrators},
\newblock \bibinfo{journal}{Phys. Lett. A} \bibinfo{volume}{150}
  (\bibinfo{year}{1990}) \bibinfo{pages}{262--268}.
\bibitem[{McLachlan and Quispel(2002)}]{mclachlan:02}
\bibinfo{author}{R.~McLachlan}, \bibinfo{author}{G.~Quispel},
\newblock \bibinfo{title}{Splitting methods},
\newblock \bibinfo{journal}{Acta Numer.} \bibinfo{volume}{11}
  (\bibinfo{year}{2002}) \bibinfo{pages}{341--434}.
\bibitem[{Cheng and Knorr(1976)}]{cheng:1976}
\bibinfo{author}{C.~Cheng}, \bibinfo{author}{G.~Knorr},
\newblock \bibinfo{title}{The integration of the {V}lasov equation in
  configuration space},
\newblock \bibinfo{journal}{J. Comput. Phys.} \bibinfo{volume}{22}
  (\bibinfo{year}{1976}) \bibinfo{pages}{330--351}.
\bibitem[{Hairer and Wanner(1996)}]{hairer2}
\bibinfo{author}{E.~Hairer}, \bibinfo{author}{G.~Wanner},
  \bibinfo{title}{{Solving Ordinary Differential Equations II: Stiff and
  Differential-Algebraic Problems}}, \bibinfo{publisher}{Springer-Verlag,
  Berlin}, \bibinfo{edition}{2nd} edition, \bibinfo{year}{1996}.
\bibitem[{Spreiter and Walter(1999)}]{spreiter1999}
\bibinfo{author}{Q.~Spreiter}, \bibinfo{author}{M.~Walter},
\newblock \bibinfo{title}{{Classical molecular dynamics simulation with the
  Velocity Verlet algorithm at strong external magnetic fields}},
\newblock \bibinfo{journal}{J. Comp. Phys.} \bibinfo{volume}{152}
  (\bibinfo{year}{1999}) \bibinfo{pages}{102--119}.
\bibitem[{Kretzschmar(2000)}]{kretzschmar2000}
\bibinfo{author}{M.~Kretzschmar},
\newblock \bibinfo{title}{{Particle motion in a Penning trap}},
\newblock \bibinfo{journal}{European J. Phys.} \bibinfo{volume}{12}
  (\bibinfo{year}{2000}) \bibinfo{pages}{240}.
\bibitem[{Blanchet(2001)}]{blanchet:2001}
\bibinfo{author}{L.~Blanchet},
\newblock \bibinfo{title}{On the two-body problem in general relativity},
\newblock \bibinfo{journal}{{C.R. Acad. Sci. Paris, Ser. IV}}
  \bibinfo{volume}{22} (\bibinfo{year}{2001}) \bibinfo{pages}{1343--1352}.
\bibitem[{Callahan et~al.(1993)Callahan, Senechal, O'Shea, Polachek, and
  Hoffman}]{callahan93}
\bibinfo{author}{J.~Callahan}, \bibinfo{author}{L.~Senechal},
  \bibinfo{author}{D.~O'Shea}, \bibinfo{author}{H.~Polachek},
  \bibinfo{author}{K.~Hoffman}, \bibinfo{title}{{Calculus in Context}},
  \bibinfo{publisher}{http://www.math.smith.edu/Local/cicintro/book.pdf},
  \bibinfo{year}{1993}.

\end{thebibliography}
    \bibliographystyle{model1-num-names}

\end{document}